\documentclass[oneside,english]{amsart}
\usepackage[T1]{fontenc}
\usepackage[latin9]{inputenc}
\usepackage{textcomp}
\usepackage{amstext}
\usepackage{amsthm}
\usepackage{amssymb}

\makeatletter
\numberwithin{equation}{section}
\numberwithin{figure}{section}
\numberwithin{table}{section}
\theoremstyle{plain}
\newtheorem{thm}{\protect\theoremname}[section]
\theoremstyle{definition}
\newtheorem{defn}[thm]{\protect\definitionname}
\theoremstyle{plain}
\newtheorem{lem}[thm]{\protect\lemmaname}

\makeatother

\usepackage{babel}
\providecommand{\definitionname}{Definition}
\providecommand{\lemmaname}{Lemma}
\providecommand{\theoremname}{Theorem}

\begin{document}
\title{MULTIDIMENSIONAL CENTRAL SETS THEOREM NEAR ZERO}
\author{Anik Pramanick, Md Mursalim Saikh}
\email{pramanick.anik@gmail.com}
\address{Department of Mathematics, University of Kalyani, Kalyani, Nadia-741235,
West Bengal, India}
\email{mdmsaikh2016@gmail.com}
\address{Department of Mathematics, University of Kalyani, Kalyani, Nadia-741235,
West Bengal, India}
\begin{abstract}
In \cite{key-1} Beiglböck gave a Multidimension Central sets theorem.
Recently, \cite{key-9} extended this result for polynomials. They proved
Multidimensional Polynomial Central sets theorem. Earlier, Hindman
and Leader introduced the near zero concept and proved the Central
sets theorem near 0 in \cite{key-10}. In this article, we generalize the
Multidimensional Central sets theorem for near 0.
\end{abstract}

\maketitle

\section{Introduction}

The concept of Near $0$ was first introduced by Hindman and Leader
in \cite{key-10} in $1999$. They proved the near zero version of
Central sets theorem in that article. Before going to Central sets,
here we first discuss the concept of near $0$. Here in near zero
we study the Ramsey theoretic results in the real interval $\left(0,1\right)$.
Surprisingly, the algebraic structure of $\beta\mathbb{R}_{d}$ helps
to study the Ramsey theoretic results of $\mathbb{R}$ with the usual
topology. Consider the semigroup $\left(\left(0,1\right),.\right)$
and define

\[
0^{+}=\cap_{\epsilon>0}Cl_{\beta\left(0,1\right)_{d}}\left(0,\epsilon\right)
\]

By $\beta\left(0,1\right)_{d}$ we mean Stone-\v{C}ech compactification
of $\left(0,1\right)$ with discrete topology. For detail discussion
about Stone-\v{C}ech compactification of semigroup , readers are
requested to go through \cite{key-12}. 

Now one can show $0^{+}$ is a two sided ideal $(\beta\left(0,1\right)_{d},.)$.
So it contains the smallest ideal of $\left(\beta\left(0,1\right)_{d},.\right)$.
So we can use those known properties of that smallest ideal. Mainly
in this article we will deal with both $\beta S$ and $\beta S_{d}$,
where $S$ is the dense semigroup of $\left(\left(0,\infty\right),+\right)$.

On the other hand, $0^{+}$ is a subsemigroup of $\left(\beta\mathbb{R}_{d},+\right)$but
not an ideal of $\left(\beta\mathbb{R}_{d},+\right)$.

Now we will focus on Central sets theorem. Furstenberg first introduce
Central sets theorem in $1981$. He defined Central sets for $K(\beta\mathbb{N})$
but here we define $Central\text{ }sets$ for an arbitrary semigroup.
\begin{defn}
The set $A$ is $central$ in $S$ if and only if there is some idempotent
$p$ in $K\left(\beta S\right)$ such that $A\in p.$ Here $K\left(\beta S\right)$
is the smallest (two-sided) ideal of $\beta S$. 
\end{defn}

Here we mention the Central sets theorem proposed by him.
\begin{thm}
Let $l\in\mathbb{N}$ and for each $i\in\left[l\right],$ let $\left(y_{i,n}\right)_{n=1}^{\infty}$be
a sequence in $\mathbb{Z}$. Let $C$ be a $central$ subset of $\mathbb{N}.$
Then there exists sequences $\left(a_{n}\right)_{n=1}^{\infty}$in
$\mathbb{N}$ and $\left(H_{n}\right)_{n=1}^{\infty}$ in $\mathcal{P}_{f}\left(\mathbb{N}\right)$
such that

$\left(1\right)$ for all $n,\max H_{n}<\min H_{n+1}$ and

$\left(2\right)$ for all $F\in\mathcal{P}_{f}\left(\mathbb{N}\right)$
and all $i\in\left[l\right],\sum_{n\in F}\left(a_{n}+\sum_{t\in H_{n}}y_{i,t}\right)\in C$.
\end{thm}

After few years in 1990 V. Bergelson and N. Hindman proved a different
but an equivalent version of the central set theorem for commutative
semigroup.
\begin{thm}
Let $\left(S,+\right)$ be a commutative semigroup. Let $l\in\mathbb{N}$
and for each $i\in\left\{ 1,2,...,l\right\} $, let $\left(y_{i,n}\right)_{n=1}^{\infty}$be
a sequence in $S$. Let $C$ be a central subset of $S$. Then there
exist sequences $\left(a_{n}\right)_{n=1}^{\infty}$in $S$ and $\left(H_{n}\right)_{n=1}^{\infty}$
in $\mathcal{P}_{f}\left(\mathbb{N}\right)$ such that

$\left(1\right)$ for all $n,\text{\ensuremath{\max H_{n}}}<\min H_{n+1}$
and 

$\left(2\right)$ for all $F\in\mathcal{P}_{f}\left(\mathbb{N}\right)$
and all $f:F\to\left\{ 1,2,...,l\right\} ,$
\[
\sum_{n\in F}\left(a_{n}+\sum_{t\in H_{n}}y_{f\left(i\right),t}\right)\in C.
\]
\end{thm}

Later many generalization of this theorem was done. But we are interested
in two of such generalizations, one is by Beiglböck \cite{key-1}
and another by Hindman and Leader \cite{key-10}. Beiglböck extended
the Central sets theorem for Multidimension. Here is the theorem proved
by Beiglböck.
\begin{thm}
Let $(S,\cdot)$ be a commutative semigroup and assume that there
exists a non principal minimal idempotent in $\beta S$. For each
$l\in\mathbb{N}$, let $\langle y_{l,n}\rangle_{n=0}^{\infty}$ be
a sequence in $S$. Let $k,r\geq1$ and let $[S]^{k}=\cup_{i=1}^{r}A_{i}$.
There exist $i\in\left\{ 1,2,\ldots,r\right\} $, a sequence $\left(x_{n}\right)_{n=1}^{\infty}$
in $S$ and a sequence $\alpha_{0}<\alpha_{1}<\ldots$ in $\mathcal{P}_{f}(\omega)$
such that for each $g\in\Phi$,
\[
\left[\text{FP}\left(\left\langle x_{n}\prod_{t\in\alpha_{n}}y_{g(n),t}\right\rangle _{n=0}^{\infty}\right)\right]_{<}^{k}\subseteq A_{i}.
\]
Where
\[
\begin{array}{c}
\left[FP\left(\left\langle x_{n}\prod_{t\in\alpha_{n}}y_{g(n),t}\right\rangle _{n=0}^{\infty}\right)\right]_{<}^{k}=\\
\left\{ \left\{ \prod_{t\in\alpha_{1}}x_{t},\ldots,\prod_{t\in\alpha_{k}}x_{t}\right\} :\alpha_{1}<\ldots<\alpha_{k}\in\mathcal{P}_{f}\left(\omega\right),\omega=\mathbb{N}\cup\left\{ 0\right\} \right\} 
\end{array}
\]
and $\Phi=\left\{ f\text{ }|\text{ }f:\omega\to\omega,\text{ }f\left(n\right)\leq n,\text{ }n\in\omega\right\} .$ 
\end{thm}

Recently Goswami and Patra \cite{key-9} extended the Polynomial Central
sets theorem for Multidimensions.On the other hand Hindman and Leader
gave a near zero version of Central sets theorem in \cite{key-10}.
\begin{thm}
\label{CST}Let $\left(S,+\right)$ be dense semigroup of $\left(\left(0,\infty\right),+\right)$
and let $A\subseteq S$ be a central set near $0$ in $S$ and , let
$Y=\left(\ensuremath{\left(y_{i,t}\right)_{t=1}^{\infty}}\right)_{i=1}^{\infty}\in\mathcal{Y}$
.Then there exists sequences $\left(a_{n}\right)_{n=1}^{\infty}$
in $S$ and a sequence $\left(H_{n}\right)_{n=1}^{\infty}$in $\mathcal{P}_{f}\left(\mathbb{N}\right)$
such that

(a) for all $n\in\mathbb{N}$, $a_{n}<\frac{1}{n}$ and $\max\text{ }H_{n}<\min\text{ }H_{n+1}$
and

(b) such that for each $f\in\Phi$, 

\[
\text{\ensuremath{FS\left(\left(a_{n}+\sum_{t\in H_{n}}y_{f\left(n\right),t}\right)_{n=1}^{\infty}\right)\subseteq A}}.
\]
\end{thm}

In this article, We will give the Multidimension version of this theorem.

\section{Preliminaries}

An Ultrafilter $p$ in $S$ is a non empty collection of subset of
$S$ satisfying the following conditions

(i) $\phi\notin p.$

(ii) If $A\in p$ then for $A\subseteq B,B\in p$

(iii) If $A,B\in p\text{ then }A\cap B\in p$

(iv) For $k\in\mathbb{N}$ and $S=\cup_{i=1}^{k}A_{i}$then there
exists $i\in\left\{ 1,2,...,k\right\} $such that $A_{i}\in p.$ 

The set of all ultrafilters on $S$ where $\left(S,\cdot\right)$
is a discrete semigroup is denoted by $\beta S$. Then $\left\{ \overline{A}:A\subseteq S\right\} $,
where $\overline{A}=\left\{ p\in\beta S:A\in p\right\} $ forms a
closed basis for the toplogy on $\beta S$. With this topology $\beta S$
becomes a compact Hausdorff space in which $S$ is dense, called the
Stone-\v{C}ech compactification of $S$. The operation of $S$ can
be extended to $\beta S$ making $\left(\beta S,\cdot\right)$ a compact,
right topological semigroup with $S$ contained in its topological
center. That is, for all $p\in\beta S$ the function $\rho_{p}:\beta S\to\beta S$
is continuous, where $\rho_{p}(q)=q\cdot p$ and for all $x\in S$,
the function $\lambda_{x}:\beta S\to\beta S$ is continuous, where
$\lambda_{x}(q)=x\cdot q$. For $p,q\in\beta S$ and $A\subseteq S$,
$A\in p\cdot q$ if and only if $\left\{ x\in S:x^{-1}A\in q\right\} \in p$,
where $x^{-1}A=\left\{ y\in S:x\cdot y\in A\right\} $. one can see
\cite{key-12} for an elementary introduction to the semigroup $\left(\beta S,\cdot\right)$
and its combinatorial applications. An element $p\in\beta S$ is called
idempotent if $p\cdot p=p$. Here we will work for those dense subsemigroup
of $\left(\left(0,1\right),\cdot\right)$, which are dense subsemigroup
of $\left(\left(0,\infty\right),+\right).$

$0^{+}$ is two sided ideal of $\left(\beta\left(0,1\right)_{d},\cdot\right)$,
so contains the smallest ideal. It is also a subsemigroup of $\left(\beta\mathbb{R}_{d},+\right)$.
As a compact right topological semigroup, $0^{+}$ has a smallest
two sided ideal. Let $K\left(0^{+}\right)$ is the smallest ideal
contained in $0^{+}$ . Central Sets near zero are the elements from
the idempotent in $K\left(0^{+}\right)$. 

\section{Multidimensional Central sets theorem near zero}

With the help of this following two lemmas we prove the main theorem.
Let define some useful notations

(i) By $\left[A\right]^{k}$we denote the collection of all subset
of $A$ with cardinality $k.$

(ii) $A^{*}=\left\{ x\in A:x^{-1}A\in p,\text{ where }p\text{ is an idempotent ultrafilter and }A\in p\right\} .$
\begin{lem}
\label{lemma1} Let $\left(S,+\right)$ be a dense semigroup of $\left(\left(0,\infty\right),+\right)$
and $p\in K$, let $k,r\geq1$ and for $\epsilon>0,\text{ }\left[S\cap\left(0,\epsilon\right)\right]^{k}=\cup_{i=1}^{r}A_{i}$.
For each $i\in\left\{ 1,2,\dots,r\right\} $, each $t\in\left\{ 1,2,\dots,k\right\} $
and each $E\in[S\cap\left(0,\epsilon\right)]^{t-1}$, define $B_{t}\left(E,i\right)$
by downward induction on $t$:
\begin{enumerate}
\item For $E\in[S\cap\left(0,\epsilon\right)]^{k-1},$$\text{ }B_{k}\left(E,i\right)=\left\{ y\in S\cap\left(0,\epsilon\right)\setminus E:E\cup\left\{ y\right\} \in A_{i}\right\} $.
\item For $1\leq t<k$ and $E\in[S\cap\left(0,\epsilon\right)]^{t-1}$,
\[
B_{t}\left(E,i\right)=\left\{ y\in S\cap\left(0,\epsilon\right)\setminus E:B_{t+1}\left(E\cup\left\{ y\right\} ,i\right)\in p\right\} .
\]
\end{enumerate}
Then there exists some $i\in\left\{ 1,2,\dots,r\right\} $ such that
$B_{1}\left(\emptyset,i\right)\in p.$
\end{lem}

\begin{proof}
The proof is by induction. Notice that for each $E\in[S\cap\left(0,\epsilon\right)]^{k-1},$
$S\cap\left(0,\epsilon\right)=E\cup\bigcup_{i=1}^{r}B_{k-1}\left(E,i\right)$
so there exists $i\in\left\{ 1,2,...,r\right\} $ such that $B_{k}\left(E,i\right)\in p$.
Let $E\in[S\cap\left(0,\epsilon\right)]^{k-2}$ and $y\in S\cap\left(0,\epsilon\right)\setminus E$
then there exists $i\in\left\{ 1,2,...,r\right\} $ such that $B_{k}\left(E\cup\left\{ y\right\} ,i\right)\in p$.
Continuing in this process we acheive $S\cap\left(0,\epsilon\right)=\phi\cup\bigcup_{i=1}^{r}B_{1}\left(\phi,i\right)$
which shows that there exists $i\in\left\{ 1,2,...,r\right\} $such
that $B_{1}\left(\phi,i\right)\in p.$
\end{proof}
For further progress we need to define some notations. The main notion
important here is the notion of Tree. In \cite{key-10} Tree is defined
for a set $A.$ According to them $T$ is a tree in $A$ if and only
if $T$ is a set of functions and for each $f\in T$, $domain\left(f\right)\in\omega$
and $range\left(f\right)\subseteq A$ and if $domain\left(f\right)=n>0$,
then $f_{|n-1}\in T$. But here we state equivalently as in \cite{key-1},
let $S$ be a set and $S^{<\omega}=\cup_{n=0}^{\infty}S^{\left\{ 0,1,...,n-1\right\} }$,
define tree by $T\subseteq S^{<\omega},$ and $T\neq\phi,$ for all
$f\in S^{<\omega},\text{ }$ $g\in T,\text{ }dom\text{ }f\subseteq dom\text{ }g,\text{ }g_{\upharpoonright dom\text{ }f}=f,f\in T$
and We will identify $f\in\left(S\right)^{\left\{ 0,1,...,n-1\right\} }$
with the touple $\left(f\left(0\right),f\left(1\right),...,f\left(n-1\right)\right).$
If $s\in\mathcal{P}_{f}\left(S\right)$ then by $f\smallfrown s=\text{\ensuremath{\left(f\left(0\right),f\left(1\right),...,f\left(n-1\right),s\right)}}$
by $T\left(f\right)=\left\{ s\in\mathcal{P}_{f}\left(S\right):f\smallfrown s\in T\right\} .$
\begin{lem}
\label{L} Let $\left(S,+\right)$ be a dense semigroup of $\left(\left(0,\infty\right),+\right)$
and let there exists a non principal minimal idempotent $p$ in $0^{+}$.
Let $k,r\geq1$ and let $[S\cap\left(0,\epsilon\right)]^{k}=\cup_{i=1}^{r}A_{i}$.
Then there exist $i\in\left\{ 1,2,\dots,r\right\} $ and $T\subseteq\left\{ S\cap\left(0,\epsilon\right)\right\} ^{<\omega}$
such that for all $f\in T$, and $\alpha_{1}<\alpha_{2}<\ldots<\alpha_{k}\subseteq domf,\alpha_{i}\in\mathcal{P}_{f}\left(\omega\right)$
one has:
\begin{enumerate}
\item $T\left(f\right)\in p$.
\item $\left\{ \sum_{t\in\alpha_{1}}f\left(t\right),\sum_{t\in\alpha_{2}}f\left(t\right),\ldots,\sum_{t\in\alpha_{k}}f\left(t\right)\right\} \in A_{i}$.
\end{enumerate}
\end{lem}

\begin{proof}
By previous lemma \ref{lemma1} we get for $i\in\left\{ 1,2,\dots,r\right\} ,\text{ }B_{1}\left(\emptyset,i\right)\in p$.
Now we construct an increasing sequences of trees $\left(T_{n}\right)_{n=0}^{\infty}$,
satisfying for each $n\geq0,\text{ }T_{n}=\left\{ f_{\upharpoonright\left\{ 1,2,...,n-1\right\} }:f\in T_{n+1}\right\} $
such that for each $f\in T_{n}$ the following holds

$\left(i\right)$ If $dom\text{ }f\subseteq\left\{ 0,1,...,n-2\right\} $
then $T_{n}\left(f\right)\in p.$

$\left(ii\right)$ If $\alpha_{1},\alpha_{2},...,\alpha_{r}\in\mathcal{P}_{f}\left(\omega\right),\text{ }r\in\left\{ 1,2,\dots,k\right\} $
satisfying $\alpha_{1}<\alpha_{2}<\ldots<\alpha_{k}\subseteq domf$
and if $x_{i}=\sum_{n\in\alpha_{i}}f\left(n\right)$ then $x_{r}\in B_{r}\left(\left\{ x_{1},x_{2},...,x_{r-1}\right\} ,i\right)^{*}.$

Here we use mathematical induction. Put $T_{0}=\left\{ \phi\right\} $
and assume now that $T_{0},T_{1},...,T_{n}$ is already defined. Fix
$f\in T_{n}$ with $dom\text{ }f=\left\{ 0,1,...,n-1\right\} .$ For
$\alpha_{1}<\alpha_{2}<\ldots<\alpha_{k}\subseteq domf$, let $x_{i}=\sum_{n\in\alpha_{i}}f\left(n\right)$.
By assumption $x_{r}\in B_{r}\left(\left\{ x_{1},x_{2},...,x_{r-1}\right\} ,i\right)$
and thus 
\[
B_{r+1}\left(\left\{ x_{1},x_{2},...,x_{r}\right\} ,i\right)\in p
\]
 for $\text{ }r\in\left\{ 1,2,\dots,k-1\right\} .$ Since $x_{r}\in B_{r}\left(\left\{ x_{1},x_{2},...,x_{r-1}\right\} ,i\right)^{*}$
we have 
\[
x_{r}^{-1}B_{r}\left(\left\{ x_{1},x_{2},...,x_{r-1}\right\} ,i\right)^{*}\in p
\]
 for $\text{ }r\in\left\{ 1,2,\dots,k\right\} $ such that indeed
$T_{n}\left(f\right)\in p.$ Using this put 
\[
T_{n+1}=T_{n}\cup\left\{ f\smallfrown t:f\in T_{n},\text{ }dom\text{\,}f=\left\{ 0,1,...,n-1\right\} ,\text{ }t\in T_{n}\left(f\right)\right\} .
\]
 It is not hard to verify that this implies that the inductive construction
can be continued: This is only interesting for $dom\text{ }f=\left\{ 0,1,...,n\right\} $
and 
\[
n\in\alpha_{r}\text{ }\left(where\text{ }r\in\left\{ 1,2,\dots,k\right\} \right).
\]
 Fix $f^{'}:\left\{ 0,1,...,n-1\right\} \to S$ such that $f^{'\smallfrown}f\left(n\right)=f.$
If $\alpha_{r}=\left\{ n\right\} ,$$x_{r}=f\left(n\right)\in T_{n}\left(f^{'}\right)\subseteq B_{r}\left(\left\{ x_{1},x_{2},...,x_{r-1}\right\} ,i\right)^{*}$
so we are done. If $\alpha_{r}=\alpha_{r}^{'}\cup\left\{ n\right\} $for
some non empty $\alpha_{r}^{'}\subseteq\left\{ 0,1,...n-1\right\} $
we have 
\[
f\left(n\right)\in T_{n}\left(f^{'}\right)\subseteq\left(\sum_{t\in\alpha_{r}^{'}}f'\left(t\right)\right)^{-1}B_{r}\left(\left\{ x_{1},x_{2},...,x_{r-1}\right\} ,i\right)^{*}
\]
and this implies $x_{r}=\sum_{t\in\alpha_{r}}f\left(t\right)\in B_{r}\left(\left\{ x_{1},x_{2},...,x_{r-1}\right\} ,i\right)^{*}$.
Finally put $T=\cup_{n=0}^{\infty}T_{n}$. Obviously $T\left(f\right)\in p$
for all $f\in T.$ Since 
\[
\sum_{n\in\alpha_{k}}f\left(n\right)\in B_{k}\left(\left\{ \sum_{t\in\alpha_{1}}f\left(t\right),\sum_{t\in\alpha_{2}}f\left(t\right),\ldots,\sum_{t\in\alpha_{k-1}}f\left(t\right)\right\} ,i\right)
\]
 for all $f\in T$ and $\alpha_{1}<\alpha_{2}<\ldots<\alpha_{k}\subseteq domf$
we see that (ii) holds.
\end{proof}
Now we proof the main result of this article.
\begin{thm}
$\left(S,+\right)$ be a dense semigroup of $\left(\left(0,\infty\right),+\right)$
and let there exists a non principal minimal idempotent $p$ in $0^{+}$and
for $\epsilon>0\text{, }\left[S\cap\left(0,\epsilon\right)\right]^{k}=\cup_{i=1}^{r}A_{i}.$For
each $l\in\mathbb{N},$let $\left(y_{l,n}\right)_{n=0}^{\infty}$
be a sequence in $S$. There exist $i\in\left\{ 1,2,\dots,r\right\} $
, a sequence $\left(a_{n}\right)_{n=0}^{\infty}$ in $S$ with $a_{n}\to0$
and $\alpha_{0}<\alpha_{1}<...$ in $\mathcal{P}_{f}\left(\omega\right)$
such that for each $g\in\Phi,$

\[
\text{\ensuremath{\left[FS\left(\left(a_{n}+\sum_{t\in\alpha_{n}}y_{g\left(n\right),t}\right)_{n=0}^{\infty}\right)\right]}}_{<}^{k}\subseteq A_{i}
\]
\end{thm}

\begin{proof}
First of all fix a minimal idempotent $p$ in $0^{+}$. Let $i\in\left\{ 1,2,\dots,l\right\} $
and for $\epsilon>0,\text{ }T\subseteq\left\{ S\cap\left(0,\epsilon\right)\right\} ^{<\omega}$
be as provided by lemma \ref{L}. We will inductively construct sequences
$\left(a_{n}\right)_{n=1}^{\infty}$ in $S$ and $\alpha_{0}<\alpha_{1}<...in\text{ }\mathcal{P}_{f}\left(\omega\right)$
such that for all $n\in\mathbb{N}$ and all $g\in\Phi$
\[
\left(\begin{array}{c}
a_{0}+\sum_{t\in\alpha_{0}}y_{g\left(0\right),t},a_{1}+\sum_{t\in\alpha_{1}}y_{g\left(1\right),t},\\
\ldots,a_{n-1}+\sum_{t\in\alpha_{n-1}}y_{g\left(n-1\right),t}
\end{array}\right)\in T..........................................\left(i\right)
\]

By the properties of $T$ this is sufficient to proof the theorem.

Assume that $a_{0},a_{1},...,a_{n-1}\in S$ and $\alpha_{0}<\alpha_{1}<...<\alpha_{n-1}\in\mathcal{P}_{f}\left(\omega\right)$
have already been constructed such that $\left(i\right)$ is true
for all $g\in\Phi.$ We have
\[
G_{n}=\cap_{g\in\Phi}T\left(\left(a_{0}+\sum_{t\in\alpha_{0}}y_{g\left(0\right),t},a_{1}+\sum_{t\in\alpha_{1}}y_{g\left(1\right),t},...,a_{n-1}+\sum_{t\in\alpha_{n-1}}y_{g\left(n-1\right),t}\right)\right)\in p.
\]
Let $m=\max\text{ }\alpha_{n-1}.$ Applying Theorem \ref{CST} for
$G_{n}$ and the sequences $\left(y_{0,n}\right)_{n=m}^{\infty}$,
$\left(y_{1,n}\right)_{n=m}^{\infty}$,..., $\left(y_{n,k}\right)_{n=m}^{\infty}$we
find $a_{n}\in S$ and $\alpha_{n}\in\mathcal{P}_{f}\left(\omega\right),$
such that 
\[
a_{n}+\sum_{t\in\alpha_{n}}y_{0,t},a_{n}+\sum_{t\in\alpha_{n}}y_{1,t},...,a_{n}+\sum_{t\in\alpha_{n}}y_{n,t}\in G_{n}
\]
 taking $f\left(n\right)=n\text{ for all }n\in\omega$ 

So for all $g\in\Phi,$$\left(a_{0}+\sum_{t\in\alpha_{0}}y_{g\left(0\right),t},a_{1}+\sum_{t\in\alpha_{1}}y_{g\left(1\right),t},...,a_{n}+\sum_{t\in\alpha_{n}}y_{g\left(n\right),t}\right)\in T.$

Now by Lemma \ref{L} the result follows.
\end{proof}
$\vspace{0.3in}$

\textbf{Acknowledgment:}The first author acknowledge the Grant CSIR-UGC
NET fellowship with file No. 09/106(0202)/2020-EMR-I. The second author
acknowledge the support from University Research Scholarship of University
of Kalyani with id-1F-7/URS/Mathematics/2023/S-502. They also grateful
to their supervisor Prof. Dibyendu De for his valuable suggestions.

$\vspace{0.3in}$

\end{document}